\documentclass[12pt]{article}

\usepackage{amsmath}
\usepackage{amssymb}
\usepackage{titling}
\usepackage[margin= 1in]{geometry}
\usepackage{graphicx}
\usepackage{mathtools}
\usepackage{graphicx}
\usepackage{xcolor}
\usepackage{amsthm}
\usepackage[utf8]{inputenc}
\usepackage[english]{babel}
\usepackage{dsfont}
\usepackage[all,cmtip]{xy}
\usepackage{subcaption}
\usepackage{float}
\usepackage{framed,enumitem}
\usepackage[nodisplayskipstretch]{setspace}
\usepackage{array}
\usepackage[export]{adjustbox}
\usepackage{multirow}
\usepackage{pdfpages}
\usepackage{comment}
\usepackage{hyperref}
\usepackage{soul}

\newtheorem{theorem}{Theorem}[section]

\newtheorem{lemma}[theorem]{Lemma}

\theoremstyle{definition}
\newtheorem{remark}{Remark}

\begin{document}

\title{Chiral Polyhedra from $\mathrm{AGL}(1,q)$} 
 
\author{Evan Angelone\ and\ Egon Schulte\\
Northeastern University\\
Boston, Massachusetts, USA, 02115}

\maketitle 


\begin{abstract}
We present a construction of chiral and regular polyhedra from subgroups of the general affine group $\mathrm{AGL}(1,q)$ for odd prime powers $q$. In particular, we show that the full group $\mathrm{AGL}(1,q)$ occurs as the automorphism group of a chiral polyhedron of type $\{q-1, q-1\}$ when $q \equiv 1 \bmod{4}$, or types $\{q-1,(q-1)/2\}$ or $\{(q-1)/2, q-1\}$ when $q \equiv 3 \bmod{4}$, and we compute the genus in each case. We also establish that subgroups of $\mathrm{AGL}(1,q)$ cannot serve as full automorphism groups of regular polytopes of rank $3$ or higher, nor of chiral polytopes of rank $4$ or higher, demonstrating that our construction captures all polytopes that can arise from this class of affine groups.
\bigskip\medskip

\medskip

\end{abstract}

\section{Introduction}
\label{intro}

Abstract polytopes provide a combinatorial generalization of classical geometric polytopes, with their study revealing deep connections between geometry, group theory, and combinatorics. Among these structures, chiral polytopes occupy a particularly intriguing position: they exhibit full rotational symmetry but lack any reflectional symmetries. The automorphism group of a chiral polytope has two orbits on flags, with adjacent flags always lying in different orbits.

The construction and classification of chiral polytopes has attracted considerable attention since their formal introduction by Schulte and Weiss~\cite{SchWei1991,SchWei1994}, and has recently resulted in the publication of Pellicer's comprehensive ``Abstract Chiral Polytopes" book~\cite{Pel2025}. While regular polytopes have been extensively studied and many construction methods are known (see McMullen \& Schulte~\cite{McMSch2002}), chiral polytopes have remained more elusive than their regular counterparts, despite significant progress. For quite some time, finite examples were known only in ranks 3 and 4. Finite examples of rank 5 were first described in Conder, Hubard \& Pisanski~\cite{CHP2008}, and the existence of chiral polytopes of arbitrarily large rank was finally confirmed in Pellicer~\cite{Pel2010}. Extensions of polytopes have helped in constructing new finite examples of chiral polytopes. Cunningham and Pellicer~\cite{CunPel2014} established that every finite chiral $d$-polytope (necessarily with regular facets) is itself the facet of a finite chiral $(d+1)$-polytope.

There has also been significant activity in studying chiral polytopes with specific types of groups, such as simple or almost simple groups or solvable groups. For example, Conder, Hubard and O'Reilly Regueiro~\cite{CHO2024} recently showed that for all but finitely many $n$, the symmetric group $S_n$ and alternating group $A_n$ occur as automorphism groups of chiral polytopes of rank 5; this extended earlier work by the same authors, together with Pellicer, for chiral 4-polytopes~\cite{CHPO2015} (see also Zhang~\cite{Zha2023a}). Leeman and Liebeck~\cite{LeeLie2017} established that, remarkably, every finite simple group is the automorphism group of a chiral polyhedron (3-polytope), with few exceptions. Chiral polytopes with automorphism groups isomorphic (or related) to $\mathrm{PSL}(2,q)$, $\mathrm{PSL}(3,q)$ or $\mathrm{PSU}(3,q)$, have also been investigated; for example, in \cite{BreCat2021,CPS2008,LMO2017,LeeVand2023}. In Conder, Feng \& Hou~\cite{CFH2021} as well as Zhang~\cite{Zha2023b}, constructions of chiral polytopes with solvable automorphism groups are presented. Recent work also studies chiral polytopes with preassigned group orders such as $2^n$ or $2^np^k$, with $p$ an odd prime \cite{Cun2018,HZG2024}, mirroring similar efforts for regular polytopes in, for example, \cite{HFL2020,HFL2025,SchWei2006}.

In this paper, we present a construction of chiral polyhedra using the affine general linear group $\mathrm{AGL}(1, q)$ for odd prime powers $q \neq 3$. It appears that the group $\mathrm{AGL}(1, q)$ is well-suited for generating chiral polyhedra, and we describe all possible ways to derive such polyhedra. Our main contributions are as follows. In Theorem \ref{gamprops}, we provide a general construction for chiral or regular polyhedra from suitable pairs of elements in $\mathrm{AGL}(1, q)$, characterizing when the resulting automorphism group is a semidirect product involving the translation subgroup. Theorem~\ref{gampropsagl} establishes that the full group $\mathrm{AGL}(1, q)$ itself does occur as the automorphism group of a chiral polyhedron, with explicit determination of the type and genus. In particular, when $q\equiv 1 \bmod{4}$, we obtain chiral polyhedra of Schl\"afli type $\{q-1, q-1\}$ and genus $1 + q(q-5)/4$; and when $q\equiv 3 \bmod{4}$, of types $\{(q-1)/2, q-1\}$ or $\{(q-1)/2, q-1\}$ and genus $1 + q(q-7)/4$.

We also prove several results delimiting the scope of our construction. Theorem~\ref{propsubgroup} shows that there are instances where the polyhedra of Theorem~\ref{gamprops} are regular, not chiral, with their automorphism group given by a subgroup of the automorphism group $\mathrm{A\Gamma L}(1, q)$ of $\mathrm{AGL}(1, q)$. We also show that for any divisor $n$ of $q-1$ (with $n > 2$), there exist chiral or regular polyhedra whose combinatorial rotation subgroups lie in $\mathrm{AGL}(1,q)$ and are isomorphic to semidirect products of subgroups of $\mathbb{F}_q$ by cyclic groups $C_n$ (in the chiral case this is the full automorphism group). Conversely, Theorem \ref{noregs} demonstrates that subgroups of $\mathrm{AGL}(1, q)$ cannot yield regular polytopes of any rank $\geq 3$, nor chiral polytopes of rank $\geq 4$, establishing that our construction is essentially optimal for this class of groups. 

The remainder of the paper is organized as follows. Section \ref{abpo} reviews the necessary background on abstract polytopes, focusing on the characterization of regular and chiral polyhedra through their automorphism groups. Section \ref{groupagl} examines the structure of $\mathrm{AGL}(1, q)$, particularly the orders of elements and properties crucial for our construction. Section \ref{constr} contains our main results, presenting the construction method and proving the existence of the resulting chiral polyhedra.

\section{Abstract polyhedra}
\label{abpo}

We briefly review some basic notions and results about chiral and regular (abstract) polytopes following~\cite{McMSch2002,Pel2025,SchWei1991}. Our main interest is in polytopes of rank $3$, also called (abstract) polyhedra. Topologically, (locally finite) abstract polyhedra are maps on surfaces.
\smallskip
 
An ({\em abstract\/}) {\em $n$-polytope\/} is a partially ordered set $\mathcal{P}$ of rank $n$. There are elements of each rank $j=-1,0,\ldots,n$. The elements of rank $j$ are called the {\em $j$-faces\/} of~$\mathcal{P}$, or the {\em vertices}, {\em edges\/}, or {\em facets\/}, if $j = 0$, $1$, or $n-1$, respectively. An $n$-polytope $\mathcal{P}$ has two \textit{improper} faces: a minimum face $F_{-1}$ (of rank $-1$) and a maximum face~$F_n$ (of rank $n$). If $n=3$ then $\mathcal{P}$ is called a {\em polyhedron\/}. In this case, if there is little possibility of confusion, we use the term ``face'' to mean ``$2$-face''. Thus a polyhedron consists of vertices, edges, and faces, as well as two improper faces.

The \textit{flags} (maximal totally ordered subsets) of an $n$-polytope $\mathcal{P}$ all contain a face of each rank $j$, including the improper faces (which often are suppressed in listing the elements of a flag). An $n$-polytope $\mathcal{P}$ is \textit{strongly flag-connected\/}; that is, given any two flags $\Phi$ and $\Psi$ of $\mathcal{P}$ there exists a sequence of flags $\Phi = \Phi_{0},\Phi_{1},\ldots,\Phi_{k} =\Psi$, such that successive flags $\Phi_{i-1}$ and $\Phi_{i}$ are {\em adjacent\/} (differ by one face), and $\Phi \cap\Psi \subseteq \Phi_{i}$ for each $i$. Moreover, $\mathcal{P}$ has what is called the diamond property:\ if $F$ and $G$ are a $(j-1)$-face and a $(j+1)$-face with $F < G$ and $0\leq j\leq n-1$, then there are exactly two $j$-faces $H$ such that $F < H < G$. The diamond property implies that, for each $j=0,\ldots,n-1$, each flag $\Phi$ is adjacent to just one other flag differing from $\Phi$ in the $j$-face; this is called the \textit{$j$-adjacent\/} flag of $\Phi$, denoted $\Phi^j$.
\smallskip

An $n$-polytope $\mathcal{P}$ is called {\em regular\/} if its ({\em automorphism\/}) {\em group\/} $\Gamma(\mathcal{P})$ is transitive on the flags of $\mathcal{P}$. If $\mathcal{P}$ is a regular $n$-polytope and $\Phi$ a fixed, {\em base\/}, flag of $\mathcal{P}$, then $\Gamma(\mathcal{P})$ is generated by involutory generators $\rho_{0},\dots,\rho_{n-1}$, where $\rho_{j}$ is the unique automorphism which fixes all faces of $\Phi$ but the $j$-face (and thus maps $\Phi$ to its $j$-adjacent flag $\Phi^j$). For a regular polyhedron there are just three generators, $\rho_{0},\rho_{1},\rho_{2}$. \smallskip

An abstract $n$-polytope $\mathcal{P}$ is {\em chiral\/} if $\Gamma(\mathcal{P})$ has two orbits on the flags, such that adjacent flags are in distinct orbits. If $\mathcal{P}$ is a chiral $n$-polytope and again $\Phi$ is a {\em base\/} flag, then $\Gamma(\mathcal{P})$ is generated by elements $\sigma_1,\ldots,\sigma_{n-1}$ where $\sigma_j$ maps $\Phi$ to $(\Phi^{j-1})^{j}$, the $j$-adjacent flag of the $(j-1)$-adjacent flag $\Phi^{j-1}$ of $\Phi$. Note that the products
\[\sigma_{i}\cdot\ldots\cdot\sigma_{j}\;\; (1\leq i<j\leq n-1)\]
all are involutions, but the generators $\sigma_i$ themselves are not. For a chiral polyhedron $\mathcal{P}$ there are just two generators, $\sigma_{1},\sigma_{2}$, with $(\sigma_1\sigma_2)^2=1$. In this case, if the faces and vertex-figures are finite, then $\sigma_1$ fixes the base face in $\Phi$ and cyclically permutes its vertices, and $\sigma_2$ fixes the base vertex and cyclically permutes the edges emanating from it; the involution $\sigma_1\sigma_2$ fixes the base edge and interchanges its vertices as well as the two faces meeting at the base edge.
\smallskip

The automorphism groups of regular or chiral polytopes also satisfy an important intersection property involving the distinguished generators. Conversely, if a group $\Gamma$ has a suitable set of generators $\rho_0,\ldots,\rho_{n-1}$ or $\sigma_1,\ldots,\sigma_{n-1}$ satisfying this intersection property, then $\Gamma$ is the automorphism group of a regular or chiral $n$-polytope. 

In the case of chiral polyhedra the intersection property takes the simple form 
\[\langle\sigma_1\rangle \cap \langle\sigma_2\rangle = \{1\}.\]
For a regular polyhedron with generators $\rho_0,\rho_1,\rho_2$, the elements $\sigma_1\coloneqq \rho_0\rho_1$ and $\sigma_2\coloneqq  \rho_1\rho_2$ generate the (combinatorial) {\it rotation subgroup\/} $\Gamma^+(\mathcal{P})$ of the automorphism group $\Gamma(\mathcal{P})$ and share many properties with the generators in the chiral case including the intersection property.
Conversely, if a group $\Gamma$ with two generators $\sigma_1,\sigma_{2}$ satisfies the relation $(\sigma_1\sigma_2)^2=1$ and has the intersection property, then $\Gamma$ is the automorphism group of a chiral polyhedron or the rotation subgroup of the automorphism group of a regular polyhedron. This polyhedron is chiral if and only there is no involutory group automorphism $\rho$ of $\Gamma(\mathcal{P})$ such that $\rho(\sigma_2)=\sigma_{2}^{-1}$ and $\rho(\sigma_1)=\sigma_{1}\sigma_2^2$, or equivalently, $\rho(\sigma_2)=\sigma_{2}^{-1}$ and $\rho(\tau)=\tau$. (For regular polyhedra, conjugation by $\rho_2$ would provide such a group automorphism.)  

For a regular or chiral polyhedron, all faces have the same number of vertices and all vertices lie in the same number of edges. If all faces of a polyhedron $\mathcal{P}$ are $s$-gons for some~$s$ and all vertices have valency $t$, then $\mathcal{P}$ is said to be of {\em type\/} $\{s,t\}$. Chiral polyhedra are maps on orientable surfaces, whereas the surfaces for regular polyhedra can be orientable or non-orientable.
\smallskip

\section{The group \texorpdfstring{$\mathrm{AGL}(1,q)$}{AGL(1, q)}}
\label{groupagl}

Let $p$ be an odd prime, $q$ be a power of $p$, and $q=p^l$. Let $q-1=2^{e}q'$ where $e\geq 1$ and $q'$ is odd, and set $r\coloneqq (q-1)/2=2^{e-1}q'$. If $q\equiv 1 \bmod 4$ then necessarily $e\geq 2$ and $r$ is even; and if $q\equiv 1 \bmod 4$ then $e=1$ and $r=q'$ is odd. 

Recall that $\mathbb{F}_q^*$ is a cyclic group of order $q-1$. For an element $g$ of a group let $o(g)$ denote its order. In particular, if $\mathbb{F}_q^*=\langle\gamma\rangle$ then 
\[o(\gamma^m)=(q-1)/\gcd(q-1,m)\] 
for each positive $m$.

We let $\rm{Aut}(\mathbb{F}_q)$ denote the group of field automorphisms of $\mathbb{F}_q$. Then $\rm{Aut}(\mathbb{F}_q)$ is a cyclic group of order $l$ generated by the Frobenius automorphism $\varphi: x\rightarrow x^p$ (see~\cite{Lang1993}). Note that, as a cyclic group of order $l$, $\rm{Aut}(\mathbb{F}_q)$ contains an involution if and only if $l$ is even. In this case, the involutory automorphism of $\mathbb{F}_q$ is given by $\varphi^{l/2}: x\rightarrow x^{p^{l/2}}\!\!=:\overline{x}$. Note that $\overline{\overline{x}}=x$ for all $x\in\mathbb{F}_q$. 

Throughout, let $\mathrm{AGL}(1,q)$ denote the affine group acting on the 1-dimensional affine space over the finite field $\mathbb{F}_q$ (the affine line with $q$ points identified with $\mathbb{F}_q$ itself). This group has order $q(q-1)$ and consists of all affine transformations $\alpha(a,b): \mathbb{F}_q \rightarrow \mathbb{F}_q$ defined by
\[ \alpha(a,b)(x) \coloneqq  ax+b,\]
for $a\in\mathbb{F}_q^*$, $b\in\mathbb{F}_q$. Note that 
\[ \alpha(a,b)\, \alpha(a',b') = \alpha(aa',ab'+b).\]
For $b\in\mathbb{F}_q$, the element $\alpha(1,b)$ of $\mathrm{AGL}(1,q)$ is called the \textit{translation by} $b$. These $q$ translations are the elements of the {\em translation subgroup\/}
\[ T\coloneqq \{\alpha(1,b)\,|\,b\in \mathbb{F}_q\} \]
of $\mathrm{AGL}(1,q)$, and have order $p$ if nontrivial. Note that $T$ is an elementary abelian normal subgroup of $\mathrm{AGL}(1,q)$ isomorphic to $\mathbb{F}_q$. It is well-known that 
\[ \mathrm{AGL}(1,q) \cong T\rtimes \mathrm{GL}(1,q)\cong \mathbb{F}_q \rtimes \mathbb{F}_q^* .\]
\smallskip 

\noindent
{\bf Remark.}\,\
{\em Alternatively, $\mathrm{AGL}(1,q)$ could be defined as the group of $2\times 2$ matrices $A$ over $\mathbb{F}_q$ of the form
\[ \left[\begin{matrix}
a&b\\
0&1
\end{matrix} \right]\]
for $a,b\in\mathbb{F}_q$, $a\neq 0$. Under the action of this group on $\mathbb{F}_q^2$, the affine line $y=1$ of $\mathbb{F}_q^2$ is invariant and 
\[A\cdot [x,1]^t = [ax+b,1]^t = [\alpha(a,b)(x),1],\]
where $t$ indicates the transpose. However, we prefer to work directly with the action of $\mathrm{AGL}(1,q)$ on the affine line $\mathbb{F}_q$, as described above.}
\medskip

The transformations in $\mathrm{AGL}(1,q)$ have the following properties.

\begin{lemma}
\label{aglprops}
Let $a\in\mathbb{F}_q^*$, $b\in\mathbb{F}_q$. Then,\\[.03in]
(a)\ $\alpha(a,b)^{-1} = \alpha(a^{-1},-a^{-1}b)$.\\[.03in]
(b)\ $\alpha(1,b)^k = \alpha(1,kb)$ for $k\geq 0$.\\[.03in]
(c)\ $\alpha(a,b)^k = \alpha(a^k, \frac{a^k-1}{a-1}b)$ for $a\neq 1$, $k\geq 0$.\\[.03in]
(d)\ The order of $\alpha(a,b)$, $a\neq 1$, in $\mathrm{AGL}(1,q)$ coincides with the order of $a$ in $\mathbb{F}_q^*$. In particular, $\alpha(a,b)$ with $a\neq 1$ has order $q-1$ if and only if $\mathbb{F}_q^*=\langle a\rangle$.\\[.03in]
(e) $\alpha(a,b)$ is an involution if and only if $a=-1$.
\end{lemma}

\begin{proof}
These properties are straightforward. Parts (c), (d) and (e) follow immediately from 
\[\alpha(a,b)^k(x) = a(a^{k-1}x+a^{k-2}b+\ldots+ab +b) +b = a^{k}x+ \tfrac{a^k-1}{a-1}b.\]
In particular, $\alpha(a,b)^k = 1$ in $\mathrm{AGL}(1,q)$ if and only if $a^k=1$ in $\mathbb{F}_q^*$. The second part of (d), as well as part (e), correspond to the cases $k=q-1$ or $k=2$, respectively.
\end{proof}
\smallskip

The product of two involutions is a translation; in fact, 
$\alpha(-1,b)\alpha(-1,c)=\alpha(1,b-c)$. This has order $p$ if $b\neq c$ or is trivial if $b=c$. Similarly, 
the product of a translation and an involution, in either order, is an involution. As a consequence, $T$ is a subgroup of index 2 in 
\begin{equation}
\label{invgen}
H\,\coloneqq \,\langle \alpha(-1,c)\,|\, c\in \mathbb{F}_q\rangle
\,=\,\{\alpha(1,b),\alpha(-1,c)\,|\, b,c\in \mathbb{F}_q\} ,
\end{equation}
which also is a normal subgroup of $\mathrm{AGL}(1,q)$. Note that $H$ is isomorphic to $T\rtimes C_2$, where any involution $\alpha(-1,c)$ can serve as the generator for $C_2$. 
\medskip

The automorphism group of $\mathrm{AGL}(1,q)$ is known to be $\mathrm{A\Gamma L}(1,q)$, the group consisting of all ``semi-affine" transformations $\gamma(a,b,\sigma): \mathbb{F}_q \rightarrow \mathbb{F}_q$ defined by
\[ \gamma(a,b)(x) \coloneqq  a x^{\sigma}+b,\]
for $a\in\mathbb{F}_q^*$, $b\in\mathbb{F}_q$, $\sigma\in\rm{Aut}(\mathbb{F}_q)$, where $x^{\sigma}:=\sigma(x)$ (see \cite[p.~98]{AK1992}). Note that $\mathrm{A\Gamma L}(1,q)$ is generated by the group of semi-linear transformations on $\mathbb{F}_q$ and the translation subgroup $T$ of $\mathrm{AGL}(1,q)$. It coincides with $\mathrm{AGL}(1,q)$ if $q$ is a prime, but is strictly larger than $\mathrm{AGL}(1,q)$ if $q$ is not a prime. 
\smallskip

Later, we require the following simple lemma about involutory group automorphisms of $\mathrm{AGL}(1,q)$. 

\begin{lemma}
\label{invgrauts}
Let $\rho:=\gamma(a,b,\sigma)$ be a group automorphism of $\mathrm{AGL}(1,q)$, where $a\in\mathbb{F}_q^*$, $b\in\mathbb{F}_q$, and $\sigma\in\rm{Aut}(\mathbb{F}_q)$ is not the identity. If $\rho$ is an involution, then $\sigma$ is also an involution (thus $l$ is even and $\sigma=\varphi^{l/2}$)
and $aa^{\sigma} = 1$ and $ab^\sigma +b=0$. \end{lemma}

\begin{proof}
Evaluate the two sides of the equation
\[x = \rho^2(x) = a(a x^{\sigma}+b)^\sigma +b = a(a^\sigma x^{\sigma^2}+b^\sigma) + b
=aa^\sigma x^{\sigma^2}+ab^\sigma +b,\]
which holds for all $x\in\mathbb{F}_q$, at $x=0$ and $x=1$ respectively. This immediately gives $ab^\sigma +b=0$ and $aa^\sigma=1$, and thus $x^{\sigma^2}=x$ for all $x\in\mathbb{F}_q$. It follows that $\sigma$ is also an involution.
\end{proof}
\bigskip

We proceed with some technical lemmas that are later needed to determine the orders of certain elements in $\mathbb{F}_q^*$.

\begin{lemma}
\label{div1}
Let $u,v$ be positive integers, and let $u$ be odd. Then 
$\gcd(2u,u+v) = 2\,\gcd(2u,v)$ if $v$ is odd, and
$\gcd(2u,u+v)=\gcd(2u,v)/2$ if $v$ is even. 
\end{lemma}

\begin{proof}
Note that $u$ is odd. If $v$ is also odd, then clearly 
\[\gcd(2u,v) = \gcd(u,v) = \gcd(u,u+v) = \gcd(2u,u+v)/2.\]
Similarly, if $v$ is even, then 
\[\gcd(2u,v) = 2\,\gcd(u,v) = 2\,\gcd(u,u+v) = 2\,\gcd(2u,u+v).\] 
\end{proof}
\smallskip

For a positive integer $m$, we let $t(m)$ denote the largest integer $l$ such that $2^l\,|\,m$. 

\begin{lemma}
\label{gcdlem}
Let $q$ be an odd prime power, let $q-1=2^{e}q'=2r$ with $e\geq 1$ and $q'$ odd, and let $1\leq m\leq q-1$. Then, 
\[
	\gcd(q - 1, r + m) = \begin{cases}
		\gcd(q - 1, m), & \text{if $e \geq 2$ and $t(m) \leq e - 2$}; \\ 
		2\gcd(q - 1, m), & \text{if $t(m) = e - 1$}, \\ 
		\gcd(q - 1, m)/2, & \text{if $t(m) \geq e$}.
	\end{cases}
\]
\end{lemma}

\begin{proof}
Our arguments below will depend on $e$ and $t(m)$. Throughout, we only need to consider divisors $d$ of $q-1$. Clearly, if $d$ is also a divisor of $r$, then $d\,|\,q-1,m$ if and only if $d\,|\,q-1, r+m$. Thus the pairs $q-1,m$ and $q-1,r+m$ share the same set of divisors when restricted to divisors of $r$. The only divisors of $q-1$ not dividing $r$ are those of the form $d=2^e s$ with $s\,|\,q'$. 

First suppose that $e\geq 2$ and $t(m)\leq e-2$. Then necessarily $q\equiv 1 \bmod 4$. We already know that the pairs $q-1,m$ and $q-1,r+m$ have the same divisors when restricted to divisors of $r$. But now $t(m)\leq e-2$, so each divisor of $q-1$ that is also a divisor of $m$ or $r+m$ is a divisor of $r$ and thus $\gcd(q-1,m)=\gcd(q-1,r+m)$, as desired.

Let $e\geq 1$ and $t(m)\geq e-1$. Then also $t(r+m)\geq e-1$. Therefore $q-1$, $m$, $r+m$ are all divisible by $2^{e-1}$, and so are their pairwise greatest common divisors. Set $u\coloneqq r/2^{e-1}=(q-1)/2^{e}=q'$ and $v\coloneqq m/2^{e-1}$. Then $u$ is odd; and $v$ is odd if $t(m)=e-1$, but $v$ is even if $t(m)\geq e$. Now it follows from Lemma~\ref{div1} that for $t(m)=e-1$,
\[\gcd(q-1,r+m)=2^{e-1}\,\gcd(2u,u+v)= 2^{e}\,\gcd(2u,v)=2\,\gcd(q-1,m),\]
and for $t(m)\geq e$,
\[\gcd(q-1,m)=2^{e-1}\,\gcd(2u,v)=2^e\,\gcd(2u,u+v)=2\,\gcd(q-1,r+m).\] 
\end{proof}
\smallskip

The previous lemma allows us to compare the orders of certain elements in $\mathbb{F}_q^*$.

\begin{lemma}
\label{neggam}
Let $\mathbb{F}_q^*=\langle\gamma\rangle$ and $1\leq m\leq q-1$. Then the orders of $\gamma^m$ and $-\gamma^{-m}$ in $\mathbb{F}_q^*$ are related as follows: 
\[
	o(-\gamma^{-m}) = \begin{cases}
		o(\gamma^m), & \text{if $e \geq 2$ and $t(m) \leq e - 2$}; \\ 
		o(\gamma^m)/2, & \text{if $t(m) = e - 1$}; \\ 
		2\, o(\gamma^m), & \text{if $t(m) \geq e$}.
	\end{cases}
\]
\end{lemma}

\begin{proof}
As $\gamma^m$ and $\gamma^{-m}$ have the same orders, we can replace $m$ by $n\coloneqq q-1-m$ and then compare $\gamma^n$ and $-\gamma^n$. Note that
\[-\gamma^{n}=(-1)\gamma^{n}=\gamma^{r}\gamma^{n}=\gamma^{r+n}\]
where again $q-1=2r$. Thus $-\gamma^{n}$ has order $(q-1)/\gcd(q-1,r+n)$. Next observe that the case assumptions for $t(m)$ in Lemma~\ref{gcdlem} are equivalent to corresponding assumptions for $t(n)$; in fact, $t(n)=t(m)$ if $t(m)\leq e-1$ (or $t(n)\leq e-1$), and $t(n)\geq e$ if and only if $t(m)\geq e$. Then $\gcd(q-1,r+n)$ and $\gcd(q-1,n)$ are related as described in Lemma~\ref{gcdlem}. It follows that the orders of $\gamma^{m}$ and $-\gamma^{-m}$ are related as stated in the lemma.
\end{proof}

\section{Construction of chiral polyhedra}
\label{constr}

In this section, we present a general construction of chiral or regular polyhedra $\mathcal{P}$ from certain subgroups of $\mathrm{AGL}(1,q)$. This amounts to describing a pair of generators $\sigma_1,\sigma_2$ for a subgroup $\Gamma$ of the automorphism group $\Gamma(\mathcal{P})$ of $\mathcal{P}$, which coincides with $\Gamma(\mathcal{P})$ if $\mathcal{P}$ is chiral or has index $2$ in $\Gamma(\mathcal{P})$ if $\mathcal{P}$ is regular; or equivalently, to describing a pair of generators $\sigma_2,\tau$ where $\tau$ is an involution constrained by $\tau=\sigma_1\sigma_2$. The polyhedron is chiral if and only if the corresponding group $\Gamma$ does not admit an involutory automorphism $\rho$ such that $\rho(\sigma_{2})=\sigma_2^{-1}$ and $\rho(\tau)=\tau$. We will assume that $q\neq 3$, as $\mathrm{AGL}(1,3)$ has order 6 and is too small for the group of a chiral or regular polyhedron.
 
The particular nature of involutions in $\mathrm{AGL}(1,q)$ severely restricts the possible choices for the involutory product $\sigma_1\sigma_2$ if $\Gamma$ is to be a subgroup of $\mathrm{AGL}(1,q)$. In fact, by Lemma~\ref{aglprops}, $\sigma_1\sigma_2$ must be a transformation of the form $\alpha(-1,c)$ given by $x \mapsto -x+c$, $c\in \mathbb{F}_q$. There are precisely $q$ such transformations in $\mathrm{AGL}(1,q)$.

We begin with 
\begin{equation}
\label{sig2}
\sigma_{2}\coloneqq \alpha(a,b),\;\tau\coloneqq \alpha(-1,c),
\end{equation}
with $a\in\mathbb{F}_q^*$, $b,c\in\mathbb{F}_q$, $a\neq \pm 1$, and importantly,
\begin{equation}
\label{condabc}
c\neq 2b/(1-a).
\end{equation}
The latter condition on $c$ will become clear in a moment but will also resurface several times later in the text. Now define 
\begin{equation}
\label{sig1}
\sigma_{1}\coloneqq \tau\sigma_2^{-1} = \alpha(-1,c)\,\alpha(a^{-1},-a^{-1}b)=\alpha(-a^{-1},a^{-1}b+c) .
\end{equation}
We let $\Gamma\coloneqq \langle\sigma_1,\sigma_2\rangle=\langle\tau,\sigma_2\rangle$. By construction, $\sigma_1\sigma_2=\tau$ is an involution. 

Note that the orders of $\sigma_2$ and $\sigma_1$ in $\Gamma$ are just the orders of $a$ and $-a^{-1}$ in $\mathbb{F}_q^*$, respectively, and are related as described in Lemma~\ref{neggam}. For example, if $a$ has order $q-1$ (that is, $a$ generates $\mathbb{F}_q^*$), then $-a^{-1}$ also has order $q-1$ if $q\equiv 1 \bmod 4$, but has order $(q-1)/2$ if $q\equiv 3 \bmod 4$. 

Next we show that $\Gamma$ has the intersection property, that is, 
$\langle\sigma_1\rangle \cap \langle\sigma_2\rangle = \{1\}$. Suppose $\sigma_1^i = \sigma_2^j$ for some $i$ and $j$, that is,
\[\alpha((-a^{-1})^{i},\tfrac{(-a^{-1})^i-1}{(-a^{-1})-1}(a^{-1}b+c))
\,=\, \alpha(a^j, \tfrac{a^j-1}{a-1}b) .\]
Then necessarily $a^j=(-a^{-1})^i$ and thus 
$\tfrac{a^{-1}b+c}{(-a^{-1})-1}=\tfrac{b}{a-1}$. 
The latter equation immediately gives $(b+ac)(a-1)=-b(1+a)$ and therefore $c=2b/(1-a)$. This is the value of $c$ that had been excluded above. Thus $\Gamma$ has the intersection property. 

Note that the intersection property indeed fails if $c=2b/(1-a)$. In fact, verifying the intersection property in this particular case just reduces to solving the equation $a^j=(-a^{-1})^i$ or, equivalently, $a^{i+j}=(-1)^i$; but this equation holds for many pairs $i,j$, for any $q$ and $a$.

Given a polyhedron already known to admit all automorphisms required for chirality, it often is a difficult task to determine whether it actually is chiral, as opposed to regular. In the current situation we can rule out that $\mathcal{P}$ is a regular polyhedron with its full automorphism group $\Gamma(\mathcal{P})$ contained in $\mathrm{AGL}(1,q)$, even though its rotation subgroup $\Gamma$ does lie fully in $\mathrm{AGL}(1,q)$. 

In fact, suppose that $\mathcal{P}$ is regular. Then its full automorphism group $\Gamma(\mathcal{P})$ is generated by involutions $\rho_0,\rho_1,\rho_2$ and contains $\Gamma$ as a subgroup (of index 1 or 2). If $\rho_0,\rho_1,\rho_2$ were to lie in $\mathrm{AGL}(1,q)$, then 
by Lemma~\ref{aglprops}, each $\rho_i$ would necessarily have to be of the form $\alpha(-1,c)$ for some $c\in\mathbb{F}_q$ and therefore
\[\Gamma\leq\Gamma(\mathcal{P})=\langle\rho_0,\rho_1,\rho_2\rangle \leq H.\]
But $\sigma_2=\alpha(a,b)\not\in H$ since $a\neq \pm 1$, which is a contradiction. Theorem~\ref{noregs} below will revisit this circle of arguments and show that no subgroup of $\mathrm{AGL}(1,q)$ can be the automorphism group of a regular polyhedron of polytope of higher rank. Note that the above reasoning also shows that if $\mathcal{P}$ is regular, then its rotation subgroup $\Gamma$ must have index 2 in the full automorphism group $\Gamma(\mathcal{P})$. Thus, $\mathcal{P}$ is orientable, regardless of whether it is chiral or regular.

Still, the previous arguments do not yet establish chirality of $\mathcal{P}$. It still could happen that there is an  involutory group automorphism $\rho$ of $\Gamma$ such that $\rho(\sigma_2)=\sigma_2^{-1}$ and $\rho(\tau)=\tau$. While this seems rather difficult to establish in the general case, we will reject this possibility at least in the case when $\Gamma=\rm{AGL}(1,q)$. 
\bigskip

\begin{theorem}
\label{gamprops}
Let $q$ be an odd prime power, $q\neq 3$, and let $\Gamma\coloneqq \langle \sigma_1,\sigma_2\rangle$ where the generators $\sigma_1$ and $\sigma_2$ are defined as in \eqref{sig2} and~\eqref{sig1}, with $a\in\mathbb{F}_q^*$, $b,c\in\mathbb{F}_q$, $a\neq \pm 1$, and $c\neq 2b/(1-a)$. Let $\mathbb{F}_q^*=\langle\gamma\rangle$ and $a=\gamma^m$, $1\leq m<q-1$. Further, let $T_\Gamma\coloneqq T\cap \Gamma$ and $H_\Gamma\coloneqq H\cap \Gamma\, (\cong T_\Gamma\rtimes C_2)$. Then, \\[.03in]
(a)\ $\Gamma$ is either the group of a chiral polyhedron or the rotation subgroup of a regular polyhedron $\mathcal{P}$, in either case with distinguished generators $\sigma_1,\sigma_2$. \\[.03in]
(b)\ $\mathcal{P}$ is of type $\{s,t\}$, where $s=o(-a^{-1})$ and $t=o(a)$ are the orders of $-a^{-1}$ and $a$ in $\mathbb{F}_q^*$, respectively. \\[.03in]
(c)\ $\Gamma \cong T_\Gamma \rtimes \langle \sigma_2 \rangle$ if $\gcd(q-1,m)\,|\,r$ and $\Gamma \cong H_\Gamma \ltimes \langle \sigma_2 \rangle$ if $\gcd(q-1,m)\!\!\not|\,r$.\\[.03in]
(d)\ $H_\Gamma = N(\tau)$, where $N(\tau)$ denotes the normal closure of $\tau=\sigma_1\sigma_2$ in $\Gamma$.
\\[.03in]
(e)\ $T_\Gamma$ contains all translations by elements of the subgroup $M_\Gamma \coloneqq  \langle\, (a^k-1)g\mid k\geq 0\rangle$, 
with $g\coloneqq c-\tfrac{2b}{1-a}$, of the additive group $\mathbb{F}_q$ whose rank is at least 
$\lceil \log_{p} ((q-1)/\gcd(q-1,m))\rceil$. \\[.03in]
(f)\ $T_{\Gamma}\cong \mathbb{F}_q$, at least if $m$ and $q-1$ are relatively prime or $q$ itself is a prime.\\[.03in]
(g)\ $\Gamma=\mathrm{AGL}(1,q)$ if and only if either $\gcd(q-1,m)=1$ (that is, $\mathbb{F}_q^*=\langle a\rangle$) or $q\equiv 3\bmod 4$ and ${\gcd(q-1,m)=2}$.
\end{theorem}

\begin{proof}
For the third part recall that $\tau=\sigma_1\sigma_2=\alpha(-1,c)\in\Gamma$. Then 
\[\sigma_2\tau\sigma_2^{-1}=\alpha(a,b)\,\alpha(-1,c)\,\alpha(a^{-1},-a^{-1}b)
=\alpha(-1,2b+ac),\]
with $2b+ac\neq c$ by our initial assumption on $c$. It follows that $\tau$ and $\sigma_2\tau\sigma_2^{-1}$ are distinct involutions and therefore their product is a translation in $T$. Note that $T\cap\langle\sigma_2\rangle=1$.

Next note that $\Gamma=N(\tau)\cdot \langle\sigma_2\rangle$, as a product of subgroups, where $N(\tau)$ is the normal closure of $\tau$ in $\Gamma$ which here is generated by the conjugates $\sigma_2^{k}\tau\sigma_{2}^{-k}$, $k\geq 0$. To show this by way of example, write an element like $\tau\sigma_2^i\tau\sigma_2^j\tau\sigma_2^l$ in the form 
\[\tau\cdot (\sigma_2^i\tau\sigma_2^{-i})\cdot(\sigma_2^{i+j}\tau\sigma_2^{-(i+j)})\cdot\sigma_2^{l+(i+j)}.\]
On the other hand, all conjugates of $\tau$ are of the form $\alpha(-1,d)$, $d\in\mathbb{F}_q$, so $N(\tau)$ lies in $H_\Gamma$. It remains to look at $H_{\Gamma}\cap\langle\sigma_2\rangle$. Note that $T_\Gamma$ is a subgroup of $H_\Gamma$ of index 2, and $H_\Gamma\cong T_{\Gamma}\rtimes \langle\tau\rangle$.

Recall that 
\[\sigma_2^k=\alpha(a,b)^k = \alpha(a^k, \tfrac{a^k-1}{a-1}b)\]
for each $k\geq 0$. Then, since $\langle\sigma_2\rangle$ intersects $T$ trivially, we have $\sigma_2^{k}\in H_\Gamma$ if and only if $a^{k}=\pm 1$ in $\mathbb{F}_q^*$; here, $a^{k}=1$ if and only if $\sigma_2^k$ is trivial. If follows that $H_{\Gamma}\cap\langle\sigma_2\rangle\neq 1$ if and only if $-1\in\langle a\rangle$ in $\mathbb{F}_q^*$. Now recall that $a=\gamma^m$ where $\mathbb{F}_q^*=\langle\gamma\rangle$. Then $a^{k}=-1$ if and only if $\gamma^{mk}=\gamma^r$, that is, if and only $mk\equiv r \bmod q-1$. 
This equation has a solution for $k$ if and only if $\gcd(q-1,m)\,|\,r$. It follows that 
$H_{\Gamma}\cap\langle\sigma_2\rangle\neq 1$ if and only if $\gcd(q-1,m)\,|\,r$. 

If $\gcd(q-1,m)\,|\,r$, then $H_\Gamma$ intersects the subgroup $\langle\sigma_2\rangle$ nontrivially but its subgroup $T_\Gamma$ does not. Thus, $\Gamma \cong T_\Gamma \ltimes \langle \sigma_2 \rangle$. Comparing group orders then shows that $N(\tau)=H_\Gamma$.

If $\gcd(m,q-1)\!\!\!\not\!|\,r$, then also $N(\tau)\cap\langle\sigma_2\rangle=1$ since $N(\tau)$ lies in $H_\Gamma$, and thus $\Gamma \cong N(\tau) \ltimes \langle \sigma_2 \rangle$. 
But the product $H_{\Gamma}\cdot\!\langle \sigma_2 \rangle$ is a subgroup of $\Gamma$ isomorphic to $H_{\Gamma}\rtimes\langle \sigma_2 \rangle$ and has the same order as $\Gamma$. Hence $N(\tau)=H_\Gamma$ and $\Gamma \cong H_\Gamma \ltimes \langle \sigma_2 \rangle$. This completes the proof of the third and fourth part of the theorem.

As a subgroup of an elementary abelian group, $T$, the group $T_\Gamma$ is itself elementary abelian. To find its rank, we first compute the generators of $N(\tau)$, and find that these are given by
\[\sigma_2^{k}\tau\sigma_2^{-k} = \alpha(-1,\tfrac{a^k-1}{a-1}\!\cdot\! 2b+a^kc),\;\,k\geq 0.\]
Then $\Gamma$ contains the translations
\[\beta_k\coloneqq \sigma_2^{k}\tau\sigma_2^{-k}\tau = \alpha(1,\tfrac{a^k-1}{a-1}\!\cdot\! 2b+(a^{k}-1)c)
=\alpha(1,(a^k-1)(c-\tfrac{2b}{1-a}))=\alpha(1,(a^k-1)g),\]
with $g=c-\tfrac{2b}{1-a}$ as defined in the theorem. Note that $g\neq 0$, by our assumption on $c$, and that $g$ does not depend on $k$. Thus $\beta_k$ is the translation by $(a^k-1)g$ for each $k$. It follows that $T_\Gamma$ contains all translation by elements from the subgroup $M_\Gamma$ of $\mathbb{F}_q$ generated by the subset
\[M'_\Gamma \coloneqq  \{(a^k-1)g\mid k\geq 0\}.\]

Now recall that $a=\gamma^m$ and has order $o(a)=(q-1)/\gcd(q-1,m)$. This shows that $M'_\Gamma$ contains $(q-1)/\gcd(q-1,m)$ translations, so $M_\Gamma$ has rank at least 
\[\lceil \log_{p} ((q-1)/\gcd(q-1,m))\rceil.\]
Hence, if $m$ and $q-1$ are relatively prime (and thus $\langle a\rangle=\mathbb{F}_q$), then the rank of $M_\Gamma$ equals the rank of $\mathbb{F}_q$ and thus $T_{\Gamma}\cong \mathbb{F}_q$. This remains true even if $\gcd(q-1,m)=2$; if fact, if $q=p^l$ with $l\geq 1$, then
$p^{l -1}(p-2)> 1$, so $(p^l -1)/2> p^{l -1}$ giving 
\[\lceil \log_{p} ((q-1)/\gcd(q-1,m))\rceil \geq l,\] 
which is just the rank of $\mathbb{F}_q$. (A similar argument shows that, for any value of $\gcd(q-1,m)$, the rank of $M_\Gamma$ is $l$ for all $q$ with $p>\gcd(q-1,m)$.) Moreover, regardless of $m$, if $q$ itself is a prime, then $T_{\Gamma}\cong \mathbb{F}_q$; this also covers the case $l =1$ in the previous claim. 

It remains to prove the final part of the theorem. If $m$ and $q-1$ are relatively prime, then by part (c), 
\[\Gamma \cong T_\Gamma \rtimes \langle \sigma_2\rangle \cong \mathbb{F}_q\rtimes \mathbb{F}_q^{*}\cong\mathrm{AGL}(1,q)\]
and thus $\Gamma=\mathrm{AGL}(1,q)$. If $q\equiv 3\bmod 4$ and $\gcd(q-1,m)=2$, then the above considerations show that still $T_\Gamma\cong\mathbb{F}_q$ and thus $H_\Gamma$ has order $2q$. Moreover, since $\gcd(q-1,m)=2$, the generator $\sigma_2$ has order $(q-1)/2$. As $\gcd(m,q-1)=2\!\!\not\!|\,r$, part (c) now gives the semidirect product $\Gamma\cong H_{\Gamma} \ltimes \langle \sigma_2\rangle$ of order $q(q-1)$, which then forces $\Gamma=\mathrm{AGL}(1,q)$.

Conversely, if $\Gamma=\mathrm{AGL}(1,q)$ then $T_\Gamma\cong\mathbb{F}_q$. Now recall the two possibilities for $\Gamma$ described in part (c) and bear in mind that $\Gamma$ has order $q(q-1)$. It follows that in the first case (when $\gcd(m,q-1)\,|\,r$) the element $\sigma_2$ has order $q-1$, that is, $o(a)=q-1$; this forces $\gcd(m,q-1)=1$. In the second case (when $\gcd(m,q-1)\!\!\not|\,r$), $H_\Gamma$ has order $2q$ and $\sigma_2$ has order $(q-1)/2=r$, giving $\gcd(m,q-1)=2$ and $q\equiv 3\bmod 4$.
\end{proof}
\bigskip

\begin{remark}
\label{dual}
Let $q$ be an odd prime power, $q\neq 3$, and let $a\in\mathbb{F}_q^*$, $b,c\in\mathbb{F}_q$, $a\neq \pm 1$, and $c\neq 2b/(1-a)$. Let $\Gamma\coloneqq \langle \sigma_1,\sigma_2\rangle$ with generators $\sigma_1,\sigma_2$ defined as in \eqref{sig2} and~\eqref{sig1}, and let $\mathcal{P}(a,b,c)$ denote the polyhedron of Theorem~\ref{gamprops} of type $\{s,t\}$ associated with $\Gamma$. Then the dual polyhedron, $\mathcal{P}(a,b,c)^*$, is given 
\[\mathcal{P}(a,b,c)^*=\mathcal{P}(-a,b+ac,c).\]
In fact, the dual polyhedron is associated with the pair of generators $\sigma_1^*,\sigma_2^*$ of $\Gamma$ given by
\[\sigma_1^{*}:=\sigma_2^{-1}=\alpha(a^{-1},-a^{-1}b),\quad 
\sigma_2^{*}:=\sigma_1^{-1} =\alpha(-a,b+ac).\]
Note that $\tau^{*}:=\sigma_1^*\sigma_2^*=(\sigma_1\sigma_2)^{-1}=\tau=\alpha(-1,c)$. The Schl\"afli type of 
$\mathcal{P}(a,b,c)^*$ is of course given by $\{t,s\}$ with $t=o(a)$ and $s=o(-a^{-1})$, which is consistent with $t=o(\sigma_1^*)=o(a^{-1})$ and $s=o(\sigma_2^*)=o(-a)$.
\end{remark}
\bigskip

A particularly interesting situation occurs when the group $\Gamma$ of Theorem~\ref{gamprops} is the full group $\mathrm{AGL}(1,q)$. As we will see below, the corresponding polyhedron will indeed be chiral in this case. However, we will show that there are instances of proper subgroups $\Gamma$ of $\mathrm{AGL}(1,q)$ where the corresponding polyhedron will be regular, not chiral.

For the proofs we require the following lemma.

\begin{lemma}
\label{conjlemma}
Let $\rho:=\gamma(g,h,\sigma)$ be an involutory group automorphism of $\mathrm{AGL}(1,q)$, where $\sigma$ is an involution in $\rm{Aut}(\mathbb{F}_q)$ and $g\in\mathbb{F}_q^*$, $h\in\mathbb{F}_q$ with $g\overline{g}= 1$, $g\overline{h}+h=0$ (see Lemma~\ref{invgrauts}). Let $\Gamma=\langle\sigma_2,\tau\rangle$ be as in Theorem~\ref{gamprops}, with $\sigma_{2}=\alpha(a,b)$ and $\tau=\alpha(-1,c)$ as in (\ref{sig2}). Suppose $\rho$ acts by conjugation as a group automorphism on $\Gamma$, that is, $\rho\,\Gamma\rho=\Gamma$. \\[.03in]
(a)\ If $\rho \sigma_2\rho = \sigma_2^{-1}$, then $\overline{a}=a^{-1}$ and thus $a^{p^{l/2} +1}=1$.\\[.03in]
(b)\ If $\rho \sigma_2\rho = \sigma_2^{-1}$ and $\rho\tau\rho=\tau$, then $\overline{a}=a^{-1}$ and $g,h$ are the unique solutions of the nonsingular linear system
\begin{equation}
\label{linsystem}
\left\{\begin{array}{rcrcr}
\overline{b} g &+& (1-\overline{a})h &=& -\overline{a}b \\
\overline{c} g &+& 2h &= &c,
\end{array}\right.
\end{equation}
where for simplicity the variables of the system were also denoted by $g$ and $h$. The solutions to \eqref{linsystem} are given by 
\begin{equation}
\label{systemsols}
g=-\overline{a}\,(\overline{D}/{D}),
\;\;h=(\overline{b}c+\overline{a}b\overline{c})/D,
\end{equation}
where $D:=2\overline{b}-\overline{c}(1-\overline{a})$ is the (non-vanishing) determinant of the linear system.
\end{lemma}

\begin{proof}
The conjugate of $\sigma_2$ by $\rho$ is given by
\begin{equation}
\label{conjrho2}
\begin{array}{rcl}
\rho\sigma_2\rho(x)\,=\,\rho\sigma_2 (g\overline{x}+h)\!\!&=&\!\!\rho(a(g\overline{x}+h)+b)\\[.03in]
\!\!&=&\!\!g\;\overline{(ag\overline{x}+ah+b)}+h \\[.03in]
&=&\!\! (g\overline{g})\overline{a}x+g\overline{a}\overline{h}+g\overline{b} +h
\,=\,\overline{a}x+g\overline{a}\overline{h}+g\overline{b} +h.
\end{array} 
\end{equation}
If this conjugate is to coincide with $\sigma_2^{-1}=\alpha(a^{-1},-a^{-1}b)$, then 
\[a^{-1}x-a^{-1}b = \overline{a}x+g\overline{a}\overline{h}+g\overline{b} +h\]
for all $x\in\mathbb{F}_q$ and therefore
\begin{equation}
\label{consigma}
\overline{a}=a^{-1},\;\, g\overline{a}\overline{h}+g\overline{b} +h=-a^{-1}b.
\end{equation}
In particular this proves the first part of the lemma. Note that $\overline{a}=a^{-1}$ just means $a^{p^{l/2} +1}=1$. Further, by assumption, $g\overline{h}+h=0$, so using $\overline{a}=a^{-1}$ allows us to rewrite the second equation in \eqref{consigma} as 
\begin{equation}
\label{chireq1}
\overline{b}g +(1-\overline{a})h=-\overline{a}b,
\end{equation}
which is a the first linear equation of \eqref{linsystem} in $g$ and $h$.

Similarly, the conjugate of $\tau$ by $\rho$ is given by
\begin{equation}
\label{conjtau}
\begin{array}{rcl}
\rho\tau\rho(x)\,=\,\rho\tau(g\overline{x}+h)\!\!&=&\!\!\rho(-(g\overline{x}+h)+c)\\[.03in]
\!\!&=&\!\! g\,\overline{(-g\overline{x}-h+c)}+h \\[.03in]
&=&\!\!-(g\overline{g})x-\!g\overline{h}+g\overline{c} +h
\,=\,-x-g\overline{h}+g\overline{c} +h.
\end{array} 
\end{equation}
If this conjugate is to coincide with $\tau$ itself, then $c=-g\overline{h}+g\overline{c} +h$ and therefore 
\begin{equation}
\label{chireq2}
\overline{c}g + 2h = c,
\end{equation}
which is the second linear equation of \eqref{linsystem} in $g$ and $h$. Here, we once again used that $g\overline{h}+h=0$. 

Therefore, if both $\rho \sigma_2\rho = \sigma_2^{-1}$ and $\rho\tau\rho=\tau$ hold, then the equations in (\ref{chireq1}) and (\ref{chireq2}) give the linear system of (\ref{systemsols}) in $g$ and $h$ with determinant $D=2\overline{b}-\overline{c}(1-\overline{a}) = \overline{2b-c(1-a)}$. 

Now recall from (\ref{sig2}) that our initial conditions on the parameters $a,b,c$ for the generators $\sigma_2,\tau$ of $\Gamma$ included the condition $c\neq 2b/(1-a)$ which ensured the intersection property for $\Gamma$. This shows that the linear system is nonsingular and has a unique solution in $g$ and $h$. Bearing in mind that $\overline{a}=a^{-1}$, the solutions can be expressed as 
\[ g=\frac{-2\overline{a}b-c(1-\overline{a})}{2\overline{b}-\overline{c}(1-\overline{a})}=
-\overline{a}\,\frac{\;\overline{D}\;}{\,D\,},
\;\;\;\; h=\frac{\overline{b}c+\overline{a}b\overline{c}}{2\overline{b}-\overline{c}(1-\overline{a})}=\frac{\overline{b}c+\overline{a}b\overline{c}}{D}.\]
Then we can check that, indeed, $g\overline{g}=1$ and $g\overline{h}+h=0$, as required by our assumptions. 
\end{proof}
\bigskip

Our next theorem deals with the most interesting case, namely when the group $\Gamma$ of Theorem~\ref{gamprops} is the full group $\mathrm{AGL}(1,q)$. This happens precisely when $\gcd(q-1,m)=1$ or $q\equiv 3\bmod 4$ and ${\gcd(q-1,m)=2}$.

\begin{theorem}
\label{gampropsagl}
Let $q$ be an odd prime power, $q\neq 3$, and let $\Gamma\coloneqq \langle \sigma_1,\sigma_2\rangle$ where the generators $\sigma_1$ and $\sigma_2$ are defined as in \eqref{sig2} and~\eqref{sig1}, with $a\in\mathbb{F}_q^*$, $b,c\in\mathbb{F}_q$, $a\neq \pm 1$, and $c\neq 2b/(1-a)$. Let $\mathbb{F}_q^*=\langle\gamma\rangle$ and $a=\gamma^m$, $1\leq m<q-1$, and let $\gcd(q-1,m)=1$ (that is, $\mathbb{F}_q^*=\langle a\rangle$) or $q\equiv 3\bmod 4$ and ${\gcd(q-1,m)=2}$. Then, \\[.03in]
(a)\ $\Gamma=\mathrm{AGL}(1,q)$ and $\Gamma$ is the group of a chiral polyhedron $\mathcal{P}$ with distinguished generators $\sigma_1,\sigma_2$. \\[.03in]
(b)\ If $\gcd(q-1,m)=1$, then $\mathcal{P}$ is either of type $\{q-1,q-1\}$ and genus $1+q(q-5)/4$ if $q\equiv 1 \bmod 4$, or type $\{(q-1)/2,q-1\}$ and genus $1+q(q-7)/4$ if $q\equiv 3 \bmod 4$.\\[.03in]
(c)\ If $q\equiv 3\bmod 4$ and ${\gcd(q-1,m)=2}$, then $\mathcal{P}$ is the dual of the polyhedron $\mathcal{P}(-a,b+ac,c)$ obtained from Theorem~\ref{gamprops} for the parameters $-a,b+ac,c$ and occurring among the polyhedra in part (b).
\end{theorem}

\begin{proof}
We know from Theorem~\ref{gamprops} that the corresponding polyhedron $\mathcal{P}$ is chiral or regular.  Moreover, by our assumptions on $m$, the group $\Gamma$ is the full group $\mathrm{AGL}(1,q)$. 

Recall from Remark~\ref{dual} that the dual of $\mathcal{P}$ is given by $\mathcal{P}(-a,b+ac,c)$, which is the polyhedron associated with the generators $\sigma_2^{-1},\sigma_1^{-1}$ of $\Gamma$. Now note $-a=(-1)a=\gamma^{r}\gamma^m=\gamma^{r+m}$, where again $r=(q-1)/2$. If $q\equiv 3 \bmod 4$ then $r$ is odd, and so $\gcd(q-1,m)=2$ if and only if $\gcd(q-1,r+m)=1$. It follows that as we pass from $a$ to $-a$, or vice versa, we are switching between polyhedra covered in different parts of the theorem, the second or third part. Thus, if $q\equiv 3 \bmod 4$, the polyhedra in third part are duals of polyhedra in the second part, and vice versa.

In the general situation, $\mathcal{P}$ is of Schl\"afli type $\{s,t\}$ with $s=o(-a^{-1})$ and $t=o(a)$. In particular, $t=(q-1)/\gcd(q-1,m)$, so $t=q-1$ if $\gcd(q-1,m)=1$ or $t=(q-1)/2$ if $\gcd(q-1,m)=2$. By Lemma~\ref{neggam}, the value of $s$ depends on $e$ as follows. First note that $e\geq 2$ if $q\equiv 1 \bmod 4$ and $e=1$ if $q\equiv 3 \bmod 4$, and therefore $t(m)=0$ if $\gcd(q-1,m)=1$, and $t(m)\geq 1$ if $\gcd(q-1,m)=2$ and $q\equiv 3 \bmod 4$. Then  Lemma~\ref{neggam} shows that 
\[s= \left\{\begin{array}{ll}
q-1=t & \mbox{if } q\equiv 1 \bmod 4\, \mbox{ and } \gcd(q-1,m)=1  \\ 
(q-1)/2=t/2 & \mbox{if } q\equiv 3 \bmod 4\, \mbox{ and } \gcd(q-1,m)=1 \\  
q-1=2t &  \mbox{if } q\equiv 3 \bmod 4\, \mbox{ and } \gcd(q-1,m)=2  \\
\end{array}\right. \]

Let $f_0$, $f_1$, and $f_2$ denote the number of vertices, edges, and faces of $\mathcal{P}$, respectively, and let $\chi = f_{0}-f_{1}+f_2$ denote the Euler characteristic and $g = 1-\chi/2$ the genus of the underlying (orientable) surface. If $\gcd(q-1,m)=1$ and $q\equiv 1 \bmod 4$, then the type is $\{q-1,q-1\}$ and $(f_0,f_1,f_2)=(q,q(q-1)/2,q)$ and thus $\chi=q(5-q)/2$ and $g=1+q(q-5)/4$. On the other hand, $\gcd(q-1,m)=1$ and $q\equiv 3 \bmod 4$, the type is $\{(q-1)/2,q-1\}$ and $(f_0,f_1,f_2)=(q,q(q-1)/2,2q)$ and thus $\chi=q(7-q)/2$ and $g=1+q(q-7)/4$. Finally, if $q\equiv 3\bmod 4$ and ${\gcd(q-1,m)=2}$, the type is $\{q-1,(q-1)/2\}$ and 
$(f_0,f_1,f_2)=(2q,q(q-1)/2,q/2)$ and thus $\chi=q(7-q)/2$ and $g=1+q(q-7)/4$. In the latter case the polyhedra are the duals of the polyhedra in the second case.

It remains to prove that $\mathcal{P}$ is chiral. We need to reject the possibility that there is an involutory group automorphism $\rho$ of $\Gamma$ such that $\rho(\sigma_2)=\sigma_2^{-1}$ and $\rho(\tau)=\tau$. 
But now $\Gamma=\mathrm{AGL}(1,q)$, so its automorphism group is given by $\mathrm{A\Gamma L}(1,q)$. 

Now suppose to the contrary that $\rho$ exists and $\mathcal{P}$ is regular. Then $\Gamma(\mathcal{P})$ is isomorphic to the subgroup $\langle\Gamma,\rho\rangle$ of $\mathrm{A\Gamma L}(1,q)$. We already explained earlier that $\rho$ cannot lie in $\mathrm{AGL}(1,q)$ itself, since otherwise $\Gamma(\mathcal{P})$ would lie in $\mathrm{AGL}(1,q)$. Thus $\rho$ is an involution $\gamma(g,h,\sigma)$ of $\mathrm{A\Gamma L}(1,q)$ of the form described in Lemma~\ref{invgrauts}; that is, $g\in\mathbb{F}_q^*$, $h\in\mathbb{F}_q$, $gg^{\sigma} = 1$,  $gh^\sigma +h=0$, and $\sigma\in\rm{Aut}(\mathbb{F}_q)$ is an involution. The latter condition is already a contradiction if $q=p^l$ with $l$ odd; in fact, recall that $\mathbb{F}_q^*$ admits involutory field automorphisms only if $q=p^l$ with $l$ even.

Now recall that, by assumption, either $\gcd(q-1,m)=1$ (so $a$ has order $p^{l}-1$) or $q\equiv 3\bmod 4$ and ${\gcd(q-1,m)=2}$. If $q\equiv 3\bmod 4$, we necessarily have $q=p^l$ with $l$ odd, so the latter case has already been settled. (Alternatively, we could appeal to duality.) In the former case where $a$ generates $\mathbb{F}_q^*$, we also arrive at a contradiction; in fact, from the first part of Lemma~\ref{conjlemma}, $a^{p^{l/2} +1}=1$, so the order of $a$ must be strictly smaller than $p^{l}-1=(p^{l/2}+1)(p^{l/2}-1)$. Note that our arguments only appeal to the first part of Lemma~\ref{conjlemma}.
\end{proof}
\medskip

For the two smallest prime powers $q=5$ and $q=7$, Theorem~\ref{gampropsagl} gives two familiar chiral torus maps:\ $\{4,4\}_{(1,2)}$ with automorphism group $\mathrm{AGL}(1,5)\cong C_5\rtimes C_4$ of order 20, and $\{3,6\}_{(1,2)}$ and $\{6,3\}_{(1,2)}$ with automorphism groups $\mathrm{AGL}(1,7)\cong C_7\rtimes C_6$ of order 42. 
\bigskip

There are many instances of proper subgroups $\Gamma$ where the polyhedron $\mathcal{P}$ of Theorem~\ref{gamprops} is indeed regular. The statement of Lemma~\ref{conjlemma} hints at certain possibilities. The following theorem covers all possible cases where the required group automorphism $\rho$ of $\Gamma$ is determined by an involutory element of $\mathrm{A\Gamma L}(1,q)$. However, in general there may be other ways in which group automorphisms of $\Gamma$ arise.

\begin{theorem}
\label{propsubgroup}
Let $q$ be an odd prime power, $q=p^l$ with $l$ even, and let $\Gamma\coloneqq \langle \sigma_1,\sigma_2\rangle$ where the generators $\sigma_1$ and $\sigma_2$ are defined as in \eqref{sig2} and~\eqref{sig1}, with $a\in\mathbb{F}_q^*$, $b,c\in\mathbb{F}_q$, $a\neq \pm 1$, $c\neq 2b/(1-a)$. Suppose $a^{p^{l/2} +1}=1$ (thus the order $o(a)$ of $a$ in $\mathbb{F}_q^*$ divides $p^{l/2} +1$) and let again $\mathcal{P}$ denote the polyhedron associated with $\Gamma$. 
Then, \\[.03in]
(a)\ $\mathcal{P}$ is a regular polyhedron of type $\{s,t\}$, with $s=o(-a^{-1})$ and $t=o(a)$.\\[.03in]
(b)\ $\Gamma(\mathcal{P})\cong \Gamma\rtimes C_2$ and $\Gamma$ is isomorphic to the rotation subgroup of $\Gamma(\mathcal{P})$.\\[.03in]
(c)\ More explicitly, $\Gamma(\mathcal{P})$ is isomorphic to the subgroup $\Lambda:=\langle\Gamma,\rho\rangle$ of $\mathrm{A\Gamma L}(1,q)$, where $\rho:=\gamma(g,h,\sigma)$ is the uniquely defined involutory group automorphism of $\mathrm{AGL}(1,q)$ such that $\sigma$ is the involution in $\rm{Aut}(\mathbb{F}_q)$ and $g\in\mathbb{F}_q^*$, $h\in\mathbb{F}_q$ are the unique solutions to the linear system in \eqref{linsystem} given in \eqref{systemsols} (thus $g\overline{g}= 1$, $g\overline{h}+h=0$). The distinguished generators of $\Lambda$ are given by $\rho_{0}:=\tau\rho$, $\rho_{1}:=\sigma_2\rho$, $\rho_{2}:=\rho$.
\end{theorem}

\begin{proof}
Recall that $\Gamma=\langle\sigma_2,\tau\rangle$ with $\tau=\sigma_1\sigma_2$. Now, reversing the steps in the proof of Lemma~\ref{conjlemma} shows that the involution $\rho$ of $\mathrm{A\Gamma L}(1,q)$ indeed conjugates the generators $\sigma_2,\tau$ of $\Gamma$ in the correct way, namely
\[\rho \sigma_2\rho = \sigma_2^{-1},\;\; \rho\tau\rho=\tau.\] 
Thus $\mathcal{P}$ is regular. Since $\rho$ does not lie in $\mathrm{AGL}(1,q)$, the subgroup $\Gamma$ of $\Lambda$ has index 2. Moreover, $\rho_0,\rho_1,\rho_2$ are involutions and $\sigma_1=\rho_0\rho_1$, $\sigma_2=\rho_1\rho_2$, $\tau = \rho_0\rho_2$. Finally,
\[\Gamma(\mathcal{P})\cong \Lambda\cong\Gamma\rtimes\langle\rho\rangle\cong\Gamma\rtimes C_2.\]
\end{proof}
\smallskip

Note that it seems difficult to establish a general criterion for when the polyhedron associated with $\Gamma$ is chiral or regular. 

We make the following additional observation. 
\smallskip

\begin{remark}
\label{general pols}
Let $q$ be an odd prime power, $q\neq 3$, let $q-1=2r=2^eq'$ with $q'$ odd, and let $n\,|\,(q-1)$, $n\neq 1,2$. Then there exists a chiral or regular polyhedron $\mathcal{P}$ whose automorphism group contains a subgroup of $\mathrm{AGL}(1,q)$ of index $1$ or $2$ isomorphic to a semidirect product of a subgroup of $\mathbb{F}_q$ by a cyclic group $C_n$, and which is of type\\[.03in]
(a)\ $\{n,n\}$, $\{n/2,n\}$, or $\{2n,n\}$, if $q\equiv 1 \bmod 4$ and $n\equiv 0 \bmod 4$, $n\equiv 2 \bmod 4$, or $n$ is odd, respectively.\\[.03in]
(b)\ $\{n/2,n\}$ or $\{2n,n\}$, if $q\equiv 3 \bmod 4$ and $n$ is even or odd, 
respectively.\\[.03in]
This can be established as follows. Define $m\coloneqq (q-1)/n$. Then also $m\,|\,(q-1)$. We use the notation of Theorem~\ref{gamprops}. In particular, $a=\gamma^m$ has order $n$, so the corresponding polyhedron has a Schl\"afli type $\{s,t\}$ with $t=n$. By Lemma~\ref{neggam}, if $q\equiv 1\bmod 4$ and thus $e\geq 2$, then $s=o(-a^{-1})$ equals $n$ if $t(m)\leq e-2$ (that is, $n\equiv 0 \bmod 4$), or $n/2$ if $t(m)=e-1$ (that is, $n\equiv 2 \bmod 4$), or $2n$ if $t(m)=e$ (that is, $n$ is odd), respectively.
Similarly, if $q\equiv 3 \bmod 4$ and thus $e=1$, then $s$ equals $n/2$ if $m$ is odd (that is, $n$ is even), or $2n$ if $m$ is even (that is, $n$ is odd), respectively.
\end{remark}
\smallskip

Our construction of polyhedra from $\mathrm{AGL}(1,q)$ involves the parameters $a,b,c$ and can lead to multiple polyhedra sharing the same type $\{s,t\}$. It would be interesting to know when polyhedra of the same type are isomorphic. 
\medskip

While the groups $\mathrm{AGL}(1,q)$ themselves give rise to chiral polyhedra, the next theorem shows that no subgroup of $\mathrm{AGL}(1,q)$ can occur as the automorphism group of a (non-degenerate) regular polyhedron or chiral or regular polytope of higher rank. Note in this context that dihedral subgroups of $\mathrm{AGL}(1,q)$, being generated by two involutions whose product is necessarily a translation, must necessarily be isomorphic to $D_p$. Thus, in rank $2$ (where every polytope is regular), only dihedral subgroups $D_p$ of $\mathrm{AGL}(1,q)$ are automorphism groups of regular polytopes, $p$-gons in this case. 

\begin{theorem}
\label{noregs}
Let $q$ be an odd prime power. \\[.03in]
(a)\ There are no regular polytopes of rank greater than or equal to 3 whose automorphism group is isomorphic to a subgroup of $\mathrm{AGL}(1,q)$.\\[.03in]
(b)\ There are no chiral polytopes of rank greater than or equal to 4 whose automorphism group is isomorphic to a subgroup of $\mathrm{AGL}(1,q)$.\\[.03in]
(c)\ If a chiral polyhedron of type $\{s,t\}$ has an automorphism group isomorphic to a subgroup of $\mathrm{AGL}(1,q)$, then $s,t\,|(q-1)$.
\end{theorem}

\begin{proof}
We already mentioned that $\mathrm{AGL}(1,3)$ is too small a group to occur, so we may assume that $q\neq 3$.

For part (a), it suffices to prove the non-existence for rank 3. Suppose $\mathcal{P}$ is a regular polyhedron whose automorphism group $\Gamma(\mathcal{P})=\langle\rho_0,\rho_1,\rho_2\rangle$ is a subgroup of $\mathrm{AGL}(1,q)$. Then 
all of $\rho_0$, $\rho_1$, $\rho_2$, and $\rho_0\rho_2$ are involutions and must lie in the subgroup $H$ of $\mathrm{AGL}(1, q)$. On the other hand, as a product of two different involutions from $H$, the involution $\rho_0\rho_2$ must also be a translation, which is impossible.

For part (b), suppose $n\geq 4$ and $\mathcal{P}$ is a chiral $n$-polytope whose automorphism group $\Gamma(\mathcal{P})=\langle\sigma_1,\ldots,\sigma_{n-1}\rangle$ is a subgroup of $\mathrm{AGL}(1,q)$. We show that this forces $\sigma_2$ to be an involution, which is impossible. In fact, $\tau_{12}\coloneqq \sigma_1\sigma_2$, $\tau_{23}\coloneqq \sigma_2\sigma_3$ and $\tau_{123}\coloneqq \sigma_1\sigma_2\sigma_3$ are involutions in $\Gamma(\mathcal{P})$ and thus lie in $H$. Then $\sigma_3=\tau_{12}\tau_{123}$ and $\sigma_1=\tau_{123}\tau_{23}$ must be translations in $T$ and therefore the element $\sigma_{2}=\sigma_1^{-1}\tau_{123}\sigma_3^{-1}$ of $H$, being a product of two translations and one involution, must be an involution. (Note that a similar argument also applies to the rotation subgroup $\Gamma^+(\mathcal{P})$ of a regular polytope of rank at least $4$ and reproves non-existence in this case.)

For part (c), consider the generator $\sigma_2$ (say). We know that $\sigma_{2}=\alpha(a,b)$ for some $a\in\mathbb{F}_q^*$ and $b\in\mathbb{F}_q$. Now, if $a\neq 1$, then $o(\sigma_{2})=o(a)\,|\,(q-1)$ and we are done. It remains to rule out the possibility that $a=1$. In fact, in this case $\sigma_2$ would be a translation and therefore $\sigma_{1}\coloneqq \tau_{12}\sigma_{2}^{-1}$, with $\tau_{12}\coloneqq \sigma_1\sigma_2$ as before, would be an involution in $H$, which is impossible. The proof that $o(\sigma_1)\,|\,(q-1)$ is similar.
\end{proof}
\medskip

The upshot of Theorem~\ref{noregs}(c) is that our construction above represents the only possible way in which chiral polyhedra can be derived directly from $\mathrm{AGL}(1,q)$.
\smallskip

Very similar arguments to those presented in the last proof will also work for most non-chiral 2-orbit polyhedra and polytopes, as well as most alternating semiregular polytopes derived from tail-triangle C-groups (see \cite{Hub2010,MonSch2012}). Their automorphism groups cannot be subgroups of $\mathrm{AGL}(1,q)$.

\subsection*{Acknowledgment}
We are grateful to an anonymous referee for a thoughtful initial review with very valuable suggestions for improvement of the paper.


\begin{thebibliography}{100}

\bibitem{AK1992} 
E.F.~Assmus and J.D.~Key, \textit{Designs and Their Codes}, Cambridge University Press, 1992.

\bibitem{BreCat2021}
A.~Breda d’Azevedo and D.A.~Catalano. Strong map-symmetry of $\mathrm{SL}(3,K)$ and $\mathrm{PSL}(3,K)$ for every finite field $K$, \textit{J. Algebra Appl.} {\bf 20} (2021), Paper No. 2150048.

\bibitem{BJS2011} A.~Breda D’Azevedo, G.A.~Jones and E.~Schulte, Constructions of chiral polytopes of small rank, \textit{Canadian J. Math.} {\bf 63} (2011), 1254--1283.

\bibitem{CFH2021}
M.D.E.~Conder, Y.-Q.~Feng, D.-D.~Hou, Two infinite families of chiral polytopes of type $\{4, 4, 4\}$ with solvable automorphism groups, \textit{J. Algebra} {\bf 569} (2021), 713--722.

\bibitem{CHO2024} 
M.~Conder, I.~Hubard and E.~O'Reilly-Regueiro, Construction of chiral polytopes of large rank with alternating or symmetric automorphism group, \textit{Adv.\ Math.} {\bf 452} (2024), Article ID 109819.

\bibitem{CHPO2015}
M.D.E.~Conder, I.~Hubard, E.~O’Reilly Regueiro and D.~Pellicer, Construction of chiral 4-polytopes with an alternating or symmetric group as automorphism group, \textit{J. Algebraic Combin.} {\bf 42} (2015), 225--244.

\bibitem{CHP2008} 
M.~Conder, I.~Hubard and T.~Pisanski, Constructions for chiral polytopes. J. London Math. Soc. (2) {\bf 77} (2008), 115--129.

\bibitem{CPS2008}
M.D.E.~Conder, P.~Poto\v{c}nik, J.~\v{S}ir\v{a}n. Regular hypermaps over projective linear groups, \textit{J. Aust. Math. Soc.} {\bf 85} (2008), 155--175.

\bibitem{Cox1970}
H.S.M.~Coxeter, \textit{Twisted Honeycombs}, Conference Board of the Mathematical Sciences Regional Conference Series in Mathematics, 4, American Mathematical Society, Providence, RI, 1970.

\bibitem{CoxMos1980}
H.S.M.~Coxeter and W.O.J.~Moser, \textit{Generators and Relations for Discrete Groups}, Springer-Verlag, 4th edition, 1980.

\bibitem{Cun2018}
G.~Cunningham, Tight chiral polyhedra, \textit{Combinatorica} {\bf 38} (2018), 115--142.

\bibitem{CunPel2014} 
G.~Cunningham and D.~Pellicer, Chiral extensions of chiral polytopes, \textit{Discrete Math.}
{\bf 330} (2014), 51--60.

\bibitem{Hub2010}
I.~Hubard, Two-orbit polyhedra from groups, \textit{European Journal of Combinatorics} {\bf 31} (2010), 943--960.

\bibitem{HubLee2014} 
I. Hubard and D. Leemans, Chiral polytopes and Suzuki simple groups, In: Rigidity and Symmetry (eds. R.~Connelly, A. Ivić Weiss, W.~Whiteley), Fields Institute Communications,  Vol. 70, 2014, 155--175.

\bibitem{HFL2020}  D.-D. Huo, Y.-Q. Feng  and D.~Leemans, Regular polytopes of 2-power order, \textit{Discrete Comp.\ Geom.} {\bf 64} (2020), 339--346.

\bibitem{HFL2025} 
D.-D.~Huo, Y.-Q.~Feng and D.~Leemans, Regular 3-polytopes of order $2^{n}p$, \textit{J. Group Theory}, 2025.

\bibitem{HZG2024}
D.-D.~Hou, T.-T.~Zheng and R.-R.~Guo, Four infinite families of chiral 3-polytopes of type $\{4, 8\}$ with solvable automorphism groups, \textit{Discrete Math.} {\bf 347} (2024), No. 113843.

\bibitem{Lang1993}
S.~Lang, \textit{Algebra}, 3rd Edition, Addison-Wellesley Publishing Company, 1993.

\bibitem{LeeLie2017}
D.~Leemans and M.W.~Liebeck, Chiral polyhedra and finite simple groups, \textit{Bull. Lond. Math. Soc.} {\bf 49} (2017), 581--592.

\bibitem{LMO2017}
D.~Leemans, J.~Moerenhout and E.~O’Reilly Regueiro, Projective linear groups as automorphism groups of chiral polytopes, \textit{J. Geom.} {\bf 108} (2017), 675--702.

\bibitem{LeeVand2023} D.~Leemans and A.~Vandenschrick, On chiral polytopes having a group $\textrm{PSL}(3, q)$ as automorphism group, \textit{J. Lond. Math. Soc.} {\bf 106} (2022), 85--111.

\bibitem{McMSch2002}
P.~McMullen and E.~Schulte, \textit{Abstract Regular Polytopes}, Cambridge University Press, 2002.

\bibitem{MonSch2012} 
B.~Monson and E.~Schulte, Semiregular polytopes and amalgamated C-groups, \textit{Adv.\ Math.} {\bf 229} (2012), 2767--2791.

\bibitem{MonWei1997} 
B.~Monson and A. I.~Weiss, Eisenstein integers and related C-groups, \textit{Geometriae Dedicata} {\bf 66} (1997), 99--117.

\bibitem{Pel2010} 
D.~Pellicer, A construction of higher rank chiral polytopes, \textit{Discrete Math.}
{\bf 310} (2010), 1222--1237.

\bibitem{Pel2025} 
D.~Pellicer, \textit{Abstract Chiral Polytopes}, Cambridge University Press, 2025.


\bibitem{SchWei1991}
E.~Schulte and A.I.~Weiss, Chiral polytopes, In \textit{Applied Geometry and Discrete Mathematics (The Victor Klee Festschrift)} (eds. P.~Gritzmann and B.~Sturmfels), DIMACS Series in Discrete Mathematics and Theoretical Computer Science, Vol. 4, American Mathematical Society and Association of Computing Machinery, 493--516, 1991.

\bibitem{SchWei1994} E.~Schulte and A. I.~Weiss, Chirality and projective linear groups. \textit{Discrete Math.} {\bf 13} (1994), 221--261.

\bibitem{SchWei2006}
E.~Schulte and A. I.~Weiss. Problems on polytopes, their groups, and realizations, \textit{Periodica Math. Hungarica} {\bf 53} (2006), 231--255.

\bibitem{Zha2023a} 
W.-J.~Zhang, Three infinite families of chiral 4-polytopes with symmetric automorphism groups, {\it J. Algebraic Combin.} {\bf 57} (2023),  421--438.

\bibitem{Zha2023b}
W.-J.~Zhang, More on chiral polytopes of type $\{4, 4, . . . , 4\}$ with solvable automorphism groups, \textit{J. Group Theory} {\bf 26} (2023), 1231--1265.

\end{thebibliography}
\end{document}